\documentclass[a4paper]{article}

\usepackage[english]{babel}
\usepackage[utf8x]{inputenc}
\usepackage[T1]{fontenc}
\usepackage{latexsym,amssymb,amsmath,amsfonts,amscd} 

\usepackage{amsmath}
\usepackage{graphicx}
\usepackage[colorinlistoftodos]{todonotes}
\usepackage[colorlinks=true, allcolors=blue]{hyperref}

\newtheorem{theorem}{Theorem}

\newtheorem{definition}{Definition}
\newtheorem{corollary}[theorem]{Corollary}
\newtheorem{proposition}[theorem]{Proposition}
\newtheorem{example}{Example}
\newtheorem{obs}{Remark}

\newtheorem{conjectura}{Conjecture}
\newenvironment{proof}[1][Proof]{\noindent\textbf{#1\,} }{\hfill \rule{0.5em}{0.5em}\medskip}

\title{Automorphisms and superalgebra structures on the Grassmann algebra}
\author{Alan de Ara\'ujo Guimar\~aes\thanks{Supported by PhD Grant from CNPq, Brazil, and by CNPq grant No. 421129/2018-2}, Plamen Koshlukov\thanks{Partially supported by FAPESP grant No. 2014/09310-5 and by CNPq grant No. 304632/2015-5.}\\
Department of Mathematics, State University of Campinas\\
651 Sergio Buarque de Holanda, 13083-859 Campinas, SP, Brazil\footnote{A. A. Guimar\~aes' current address: Department of Mathematics, Federal University of Rio Grande do Norte, 59078-970 Natal, RN, Brazil}\\
e-mails: \texttt{alansimoes10@hotmail.com}, \texttt{plamen@ime.unicamp.br}
} 
\date{}
\begin{document}
\maketitle

\begin{abstract}
Let $F$ be a field of characteristic zero and let $E$ be the Grassmann algebra of an infinite dimensional $F$-vector space $L$. In this paper we study the superalgebra structures (that is the $\mathbb{Z}_{2}$-gradings) that the algebra $E$ admits. By using the duality between superalgebras and automorphisms of order $2$ we prove that in many cases the $\mathbb{Z}_{2}$-graded polynomial identities for such structures coincide with the $\mathbb{Z}_{2}$-graded polynomial identities of the "typical" cases $E_{\infty}$, $E_{k^\ast}$ and $E_{k}$ where the vector space $L$ is homogeneous. Recall that these cases were completely described by Di Vincenzo and Da Silva in \cite{disil}. Moreover we exhibit a wide range of non-homogeneous $\mathbb{Z}_{2}$-gradings on $E$ that are $\mathbb{Z}_{2}$-isomorphic to $E_{\infty}$, $E_{k^\ast}$ and $E_{k}$. In particular we construct a $\mathbb{Z}_{2}$-grading on $E$ with only one homogeneous generator in $L$ which is $\mathbb{Z}_{2}$-isomorphic to the natural $\mathbb{Z}_{2}$-grading on $E$, here denoted by $E_{can}$. 
\end{abstract}

\section{Introduction}

Let $F$ be a field of characteristic zero. If $L$ is a vector space over $F$ with basis $e_{1}$, $e_{2}$, \dots, the infinite dimensional Grassmann algebra $E$ of $L$ over $F$ has a basis consisting of $1$ and all monomials $e_{i_1}e_{i_2}\cdots e_{i_k}$ where $i_{1}<i_{2}<\cdots <i_{k}$, $k\geq 1$. The multiplication in $E$ is induced by the rule $e_{i}e_{j}=-e_{j}e_{i}$ for all $i$ and $j$; in particular $e_i^2=0$ for each $i$. The Grassmann algebra is the most natural example of a \textsl{superalgebra}, it is widely used in various parts of Mathematics and also in Theoretical Physics. The Grassmann algebra $E$ is one of the most important algebras satisfying a polynomial identity, also known as PI-algebras. Its polynomial identities were described by Latyshev \cite{latyshev}, and later on, in great detail,  by Krakowski and Regev \cite{KR}. It was observed by P. M. Cohn that the Grassmann algebra satisfies no standard identities. Recall that the standard polynomial $s_n$ is the alternating sum of all monomials obtained by permuting the variables in $x_1x_2\cdots x_n$. In fact the true importance of the Grassmann algebra in the area of PI algebras was revealed through the theory developed by A. Kemer in the mid eighties. Kemer proved that every associative PI-algebra over a field of characteristic zero is PI-equivalent (that is it satisfies the same polynomial identities) to the Grassmann envelope of a finite dimensional associative superalgebra, see \cite{basekemer2, basekemer}. In Kemer's structure theory for the ideals of identities of associative algebras the natural $\mathbb{Z}_{2}$-grading on $E$  was used. Here we point out that this superalgebra structure on $E$ is defined by $E_{can}=E_{(0)}\oplus E_{(1)}$. Here $E_{(0)}$ is the vector subspace of $E$ spanned by 1 and all monomials of even length and $E_{(1)}$ is spanned by all monomials of odd length. It is immediate to see that $E_{(0)}$ is the centre of $E$ while $E_{(1)}$ is the "anticommuting" part of $E$. Recall here that if $A=A_0\oplus A_1$ is a $\mathbb{Z}_2$-graded algebra then its Grassmann envelope is defined as $E(A)=A_0\otimes E_{(0)} \oplus A_1\otimes E_{(1)}$. 

Clearly the Grassmann algebra admits various other gradings. A natural and interesting problem in this direciton is to investigate the structure of group gradings on the Grassmann algebra and the corresponding graded polynomial identities. 

Let $[x_{1},x_{2}]=x_{1}x_{2}-x_{2}x_{1}$ be the commutator of $x_{1}$ and $x_{2}$. We define $[x_{1},x_{2},x_{3}]=[[x_{1},x_{2}],x_{3}]$ and so on for more variables. In other words when no additional brackets are written we assume the longer commutators left normed. As mentioned above, Latyshev, and Krakowski and Regev proved that the triple commutator $[x_{1},x_{2},x_{3}]$ forms a basis for the ordinary polynomial identities of $E$ (see \cite{latyshev, KR}). Krakowski and Regev also computed the codimension sequence of the ideal of identities of $E$. The interested reader could find more details concerning the polynomial identities of the Grassmann algebra in \cite{agpk} and the references therein.

Group gradings on algebras and the corresponding graded polynomial identities have been extensively studied in PI theory during the last three decades. The graded polynomial identities satisfied by an algebra are ``easier'' to describe than the ordinary ones. Here we recall that the polynomial identities of an associative algebra $A$ are known in very few instances. These include the Grassmann algebra $E$ (over any field), the $2\times 2$ matrix algebra \cite{razmyslovm2, drenskym2} (in characteristic 0), \cite{pkm2} (over infinite fields of characteristic different from 2), and \cite{malkuz} (over a finite field); the upper triangular matrix algebras \cite{malcev, siderov}, and the algebra $E\otimes E$ \cite{popov} in characteristic 0. 

On the other hand the gradings on matrix algebras are known \cite{bz}. The graded identities of these algebras are well understood, see for example \cite{vas1, vas2, ssa1, ssa2} for the natural gradings on the $n\times n$ matrices. The gradings and the corresponding graded identities on the upper triangular matrices are also well known \cite{vaza, opa}. The graded identities for natural gradings on classes of important algebras also have been described, see for example \cite{vinnar1, vinnar2}, and also \cite{sapk}, as well as the references in these three papers.  

The Grassmann algebra admits the natural grading by the cyclic group $\mathbb{Z}_{2}$ of order two. Its structure from the point of view of the PI theory is well known and easy to deduce, see for example \cite{GMZ}. In recent years a substantial number of papers has presented results on gradings and their graded identities for the Grassmann algebra. In all of them, two conditions have been imposed, namely:

\begin{itemize}
\item The grading group $G$ is finite. 
\item All generators $e_{1}$, $e_{2}$, \dots, $e_{n}$, \dots{} of $E$ are  homogeneous. This means that the basis of the vector space $L$ is homogeneous. 
\end{itemize}

When a grading on $E$ satisfies the latter condition it is called \textsl{homogeneous grading}. In \cite{disil} the authors studied all homogeneous superalgebra structures defined on the Grassmann algebra. These were denoted as $E_{k}$, $E_{k^\ast}$ and $E_{\infty}$ (we give, for the readers' convenience, their definitions in the next section of the paper). Their $\mathbb{Z}_{2}$-graded polynomial identities were also described in \cite{disil}. In this context the following question arises naturally:

\begin{itemize}
\item Is the list $\{E_{k}$, $E_{k^\ast}$, $E_{\infty}\}$ of the Grassmann superalgebras  complete?
\end{itemize}

Clearly one has to seek an answer to the above question up to graded isomorphisms. 

In order to look for a possible structure of a superalgebra on the Grassmann algebra (which does not fit in any one of the three cases above) it is necessary to study $\mathbb{Z}_{2}$-gradings on $E$ without the hypothesis of homogeneity. 

We refer the reader to \cite{disilplamen, centroneG} and the references therein for various results concerning group grading on the Grassmann algebra. In \cite{disilplamen} the authors approach the problem of grading on $E$ by a cyclic group of prime order, and in \cite{centroneG} the author addresses the case of finite abelian groups. 

In the present paper we shall study general structures of superalgebras on the Grassmann algebra, without supposing the condition of homogeneity of the vector space $L$. To this end we use the duality between $\mathbb{Z}_{2}$-grading and automorphisms of order $2$ on $E$. This duality is well known. It relies on  the fact that if $G$ is a finite abelian group then $G$ is isomorphic to its dual group assuming that the field is large enough. As we are interested in gradings by the group $\mathbb{Z}_2$ we need no further assumptions on the base field (apart from its characteristic being different from 2). Thus if $\varphi\in Aut(A)$ is an automorphism of an algebra $A$ of order two, that is $\varphi^2=1$ then one has a $\mathbb{Z}_2$-grading on $A$ given by $A= A_{0,\varphi}\oplus A_{1,\varphi}$. Here $A_{0,\varphi}$ and $A_{1,\varphi}$ are the eigenspaces in $A$ associated to eigenvalues 1 and $-1$ of the linear transformation $\varphi$. Reciprocally to each $\mathbb{Z}_2$-grading on $A$ one associates an automorphism of $A$ of order 2 as follows. If $A=A_0\oplus A_1$ is the $\mathbb{Z}_2$-grading the automorphism $\varphi$ is defined by $\varphi(a_0+a_1) = a_0-a_1$ for every $a_i\in A_i$, $i=0$, 1. We shall need this duality in the form of a duality between group gradings and group actions, see for example \cite{GMZ} for a discussion in the general case.

Let us fix a basis $\beta=\{e_{1},e_{2},\ldots,e_{n},\ldots\}$ of the vector space $L$ and an automorphism $\varphi\in Aut(E)$ such that $\varphi^{2}=1$, we consider the set
\[
I_{\beta}=\{n\in\mathbb{N}\mid\varphi(e_{n})=\pm e_{n}\}.
\]
There exist the following four possibilities: 
\begin{enumerate}
	\item $I_{\beta}=\mathbb{N}$. 
	\item $I_{\beta}\neq\mathbb{N}$ is infinite. 
	\item $I_{\beta}$ is finite and non empty.
\end{enumerate}

Given a basis $\beta$ of $L$ it is possible that $I_{\beta}=\emptyset $ but $I_{\beta'}\neq\emptyset $ for some other basis $\beta'$ of $L$, see Example \ref{example conjecture}. Hence the fourth possibility that we shall consider is the following.
\begin{enumerate}
	\item [4.]$I_{\gamma}=\emptyset $ for every basis $\gamma$ of the vector space $L$.
\end{enumerate}

The first possibility, when $I_{\beta}=\mathbb{N}$, corresponds to the homogeneous case. As already mentioned, it was completely described by Di Vincenzo and Da Silva in their paper \cite{disil}.

Therefore in order to construct non-homogeneous Grassmann superalgebras one is led to deal with automorphisms $\varphi$ on $E$ satisfying either one of the conditions 2, 3 or 4. In this paper we shall exhibit structures of types 2 and 3. We shall also prove that these structures are in fact equivalent to homogeneous superalgebras. In particular we shall construct one kind of a Grassmann superalgebra  where only one generator among the $e_i$ is homogeneous. Nevertheless it will turn out that such a superalgebra is  equivalent to the one coming from the natural grading  $E_{can}$.

Furthermore we shall prove that the fourth structure does not exist in quite many cases. We will not provide examples of Grassmann superalgebras of type 4 since we could not find any. In fact we have some ground to conjecture that the fourth case does not happen at all. But we have not been able to prove it yet.

\section{Preliminaries}
Let $F$ be a field and let $A$ be a unitary associative $F$-algebra. We say that $A$ is a $\mathbb{Z}_{2}$-graded algebra (or superalgebra) whenever $A=A_{0}\oplus A_{1}$ where $A_{0}$, $A_{1}$ are $F$-subspaces of $A$ satisfying $A_{i}A_{j}\subset A_{i+j}$ for $i$, $j\in\mathbb{Z}_{2}$. The vector subspace $A_{i}$ is called the {\sl $i$-homogeneous component} of $A$. If $a\in A_{i}$ we say that $a$ is homogeneous and is of (homogeneous) degree $\|a\|=i$. A vector subspace (subalgebra, ideal) $W\subset A$ is homogeneous if $W=(W\cap A_{0})\oplus (W\cap A_{1})$. 

Here we point out that we use "freely" the terms \textsl{superalgebra} and $\mathbb{Z}_2$-\textsl{graded algebra} as synonymous although this is an abuse of terminology. In the associative case they are indeed synonymous while in the nonassociative setting they are not. Indeed, a Lie or a Jordan superalgebra is not, as a rule, a Lie or a Jordan algebra. The correct setting in the general case should be as follows. Let $A=A_0\oplus A_1$ be a $\mathbb{Z}_2$-graded algebra and let $\mathfrak{V}$ be a variety of algebras (not necessarily associative). Then $A$ is a $\mathfrak{V}$-superalgebra whenever $A_0\otimes E_{(0)} \oplus A_1\otimes E_{(1)}$ is an \textsl{algebra} belonging to $\mathfrak{V}$. (We draw the readers' attention that one does not require $A\in\mathfrak{V}$.) Since we shall deal with associative algebras only such a distinction is not relevant for our purposes, and we are not going to make any difference between superalgebras and $\mathbb{Z}_2$-graded algebras.

If $A$, $B$ are superalgebras, a homomorphism $f\colon A\rightarrow B$ is a $\mathbb{Z}_{2}$-{\it graded homomorphism} if $f(A_{i})\subset B_{i}$ for all $i\in\mathbb{Z}_{2}$. When there exists a $\mathbb{Z}_{2}$-graded isomorphism between $A$ and $B$ we say that $A$ and $B$ are $\mathbb{Z}_{2}$-isomorphic or \textsl{equivalent}.

One defines a free object in the class of superalgebras by considering the free $F$-algebra over the disjoint union of two countable sets of variables, denoted by $Y$ and $Z$. We assume further that the elements of $Y$ are of degree zero and the elements of $Z$ are of degree $1$. This algebra is denoted by $F\langle Y\cup Z \rangle$. Its even part is the vector space spanned by all monomials whose degree counting only the elements of $Z$, is an even integer. The remaining monomials span the odd component. It is straightforward that $F\langle Y\cup Z\rangle$ is a free algebra in the sense that for every superalgebra $A$ and for every map $\varphi\colon Y\cup Z\to A$ such that $\varphi(Y)\subseteq A_0$ and $\varphi(Z)\subseteq A_1$ there exists unique homomorphism of $\mathbb{Z}_2$-graded algebras $F\langle Y\cup Z\rangle\to A$ that extends $\varphi$.

We say that the polynomial $f(y_{1},\ldots, y_{l},z_{1},\ldots, z_{m})\in F\langle Y\cup Z \rangle$ is a $\mathbb{Z}_{2}$-{\it graded polynomial identity} for a superalgebra $A$ if $f(a_{1},\ldots, a_{l},b_{1},\ldots, b_{m})=0$ for all substitutions such that $\| a_i\|=0$ and $\|b_j\|=1$. The set $T_{2}(A)$ of all $\mathbb{Z}_{2}$-graded polynomial identity of $A$ is a homogeneous ideal  of $F\langle Y\cup Z \rangle$. It is called the $T_{2}$-ideal of $A$, and denoted by $T_2(A)$.

The description of all structures of a superalgebra on a given algebra is an important task in Ring theory, and in particular in PI theory. As commented above, in \cite{disil} the authors considered $\mathbb{Z}_{2}$-grading on the Grassmann algebra $E$ over a field of characteristic zero such that the vector space $L$ is homogeneous in the grading. There exist three structures, namely:
\[
\|e_{i}\|_{k}=\begin{cases} 0,\text{ if }  i=1,\ldots,k\\ 
1, \text{ otherwise} 
\end{cases},
\]
\[
\|e_{i}\|_{k^\ast}=\begin{cases} 1,\text{ if }  i=1,\ldots,k\\ 
0, \text{ otherwise } 
\end{cases},
\]	
and
\[
\|e_{i}\|_{\infty}=\begin{cases} 0,\text{ if }  i\text{ is even }\\ 1, \text{ otherwise } 
\end{cases}.
\]
These provide us with three superalgebras, $E_{k}$, $E_{k^\ast}$ and $E_{\infty}$, respectively. The $\mathbb{Z}_{2}$-graded polynomial identities of $E_{k}$, $E_{\infty}$, and $E_{k^\ast}$ were also described in \cite{disil}. These structures were studied over an infinite field of positive characteristic $p>2$ in \cite{centroneP} and also over finite fields, see \cite{LFG}. From now on we shall call $E_{k}$, $E_{\infty}$ and $E_{k^\ast}$ the \textsl{homogeneous Grassmann superalgebras}.

\subsection{Gradings, automorphisms, and their duality} 
Consider $A$ an associative $F$-algebra. There exists a natural duality between $\mathbb{Z}_{2}$-gradings and automorphisms of order 2 on $A$. The duality, as commented above, is defined as follows. 

If $\varphi\in Aut(A)$ is such that $\varphi^{2}=1$ then $A_{\varphi}=A_{0,\varphi}\oplus A_{1,\varphi}$ 
where the homogeneous components are the eigenspaces corresponding to the eigenvalues $1$ and $-1$ of $\varphi$, respectively. The decomposition in a direct sum of the eigenspaces exists since $F$ is of characteristic different from 2.

The general facts about duality between gradings and actions of groups can be found, for example, in \cite[Chapter 3]{bookGB} and also in the paper \cite{GMZ}. We observe that the homogeneous $\mathbb{Z}_{2}$-gradings on $E$ correspond to the automorphisms on $E$ satisfying
\[
\varphi(e_i)=\pm e_i.
\]
If $\varphi\in Aut(E)$ with $\varphi^{2}=1$ we observe that
\[
e_{i}=(e_{i}+\varphi(e_{i}))/{2} + (e_{i}-\varphi(e_{i}))/{2}, \mbox{\textrm{ for }}i\in\mathbb{N}.
\]
Setting $a_{i}=e_{i}+\varphi(e_{i})/{2}$ we have
\begin{itemize}
	\item $\varphi(e_i)=-e_i + 2a_{i}$,
	\item $\varphi(a_i)=a_i$, that is, $a_i$ is of degree zero in the $\mathbb{Z}_{2}$-grading $E_{\varphi}$, and
	\item $\varphi(e_i - a_i)=-(e_i - a_i)$, that is, $e_i - a_i $ is of degree $1$ in the $\mathbb{Z}_{2}$-grading $E_{\varphi}$.   
\end{itemize}

\begin{definition} Assume $\varphi\in Aut(E)$ is of order two and let $E_{can}=E_{(0)}\oplus E_{(1)}$ be the natural $\mathbb{Z}_{2}$-grading on $E$. We say that $\varphi$ is of canonical type if 
	\begin{enumerate}
		\item $\varphi(E_{(0)})=E_{(0)}$;
		\item $\varphi(E_{(1)})=E_{(1)}$.
	\end{enumerate}	
\end{definition}

Concerning the canonical automorphisms it is easy to check that 		
	\begin{itemize} 
	    \item The superalgebras $E_{\infty}$, $E_{k^{\ast}}$ and $E_{k}$ correspond to automorphisms of canonical type. 
	
    	\item If $\varphi$ is an automorphism of order $2$ on $E$ we have that 
	$\varphi$ is of canonical type if only if $a_{i}\in E_{(1)}$ for all $i\in\mathbb{N}$. 
	\end{itemize}

Let us fix a basis $\beta=\{e_{1},e_{2},\ldots,e_{n},\ldots\}$ of the vector space $L$ and an automorphism $\varphi\in Aut(E)$ such that $\varphi^{2}=1$. Then $\varphi$, as a linear transformation, has eigenvalues 1 and $-1$ only, and moreover, there exists a basis of the vector space $E$ consisting of eigenvectors. (It is well known from the elementary Linear algebra that this fact does not depend on the dimension of the vector space as long as the characteristic of $F$ is different from 2.) Then $E=E(1)\oplus E(-1)$ where $E(t)$ is the eigenspace for the eigenvalue $t$ of the linear transformation $\varphi$. One considers the intersections $L(t)=L\cap E(t)$, $t=\pm 1$. Changing the basis $\beta$, if necessary, one may assume that $L(t)$ is the span of $\beta\cap L(t)$. Clearly this change of basis gives rise to a homogeneous automorphism of $E$ and we can take the composition of it and then $\varphi$. We shall assume that such a change of basis has been done. 

Denote $I_{\beta}=\{n\in\mathbb{N}\mid\varphi(e_{n})=\pm e_{n}\}$. We shall distinguish the following four possibilities: 
\begin{enumerate}
	\item $I_{\beta}=\mathbb{N}$.
	\item $I_{\beta}\neq\mathbb{N}$ is infinite. 
	\item $I_{\beta}$ is finite and non empty.
	\item $I_{\gamma}=\emptyset $ for every linear basis $\gamma$ of $L$. 
\end{enumerate}
%It is possible that $I_{\beta}=\emptyset $ and $I_{\beta'}\neq\emptyset $, for other basis $\beta'$ of $L$, see Example \ref{example conjecture}. Hence the fourth possibility that we shall consider is
%\begin{enumerate}
%	\item [4.]$I_{\gamma}=\emptyset $ for every linear basis $\gamma$ of $L$.
%\end{enumerate}
The automorphisms of type 1 have been completely described by Di Vincenzo and Da Silva \cite{disil}. Therefore we shall focus on the remaining three cases.

We shall call these automorphisms (and also the corresponding $\mathbb{Z}_{2}$-gradings), \textsl{ automorphisms} ($\mathbb{Z}_{2}$-grading) of  type 1, 2, 3, and 4, respectively. 

\section{Gradings and automorphisms of type 2} 
In this section we shall describe automorphisms of type 2. We start with the following proposition.
\begin{proposition}\label{auto tipo 2} 
	Let $\varphi\in Aut(E)$ be an automorphism such that $\varphi^{2}=1$. Suppose that $\varphi$ is an automorphism of type 2, that is,
$I_{\beta}\neq\mathbb{N}$ is infinite. Then $\varphi$ is of canonical type.  
\end{proposition}
\begin{proof}
Let $E_{\varphi}=E_{0,\varphi}\oplus E_{1,\varphi}$ be the superalgebra structure on $E$ induced by $\varphi$, we denote $J=\{j\in\mathbb{N} \mid j\notin I_{\beta}\}$. For the sake of simplicity we shall write $I=I_{\beta}$.

For each $j\in J$ we have $\varphi(e_{j})\neq\pm e_{j}$. Therefore $e_{j}$ is not homogeneous in the grading $E_{\varphi}=E_{0,\varphi}\oplus E_{1,\varphi}$. Thus we can write 
\[
e_{j}=b_{j}+c_{j},
\] 
with $b_{j}\in E_{0,\varphi}$ and $c_{j}\in E_{1,\varphi}$. Moreover we have that each summand $b_j$ and $c_j$ is non zero.

Since $\varphi(c_{j})=-c_{j}$ we have
	\[
	\varphi(e_{j}-b_{j})=-e_{j}+b_{i}
	\]
and this implies the equality $\varphi(e_{j})=-e_{j}+2b_{j}$. 
	Next we prove that each $b_{j}$ is a linear combination of monomials of odd length. To this end we write 
	\[
	b_{j}=b_{j}^{e} + b_{j}^{o}.
	\]
Here $b_{j}^{e}$ is the linear combination of the monomials of even length in $b_{j}$ and $b_{j}^{o}$ is the one consisting of all  monomials of odd length in $b_{j}$.

Choose now $t\in I$ such that $e_{t}$ does not belong to the support of the summands of $b_{j}^{e}$ (this is always possible since the set $I$ is infinite). As
	\[
	e_{j}e_{t}+e_{t}e_{j}=0
	\]
	we obtain that
	\[
	\varphi(e_{j})\varphi(e_{t})+\varphi(e_{t})\varphi(e_{j})=0.
	\]
	By $\varphi(e_{t})=\pm e_{t}$ it follows that
	\[(-e_{j}+2b_{j})e_{t}+e_{t}(-e_{j}+2b_{j})=0\]
which in turn implies
	\[
	b_{j}e_{t}+e_{t}b_{j}=0.
	\]
Hence $(b_{j}^{e} + b_{j}^{o})e_{t}+e_{t}(b_{j}^{e} + b_{j}^{o})=0$. It follows that $e_{t}b_{j}^{e}=0$ and $b_{j}^{e}=0$.
	Therefore the element $b_{j}$ is a linear combination of monomials of odd length for all $j\in J$, and the proof follows.
\end{proof}

Fix a basis $\beta$ of $L$ and let $\varphi$ be an automorphism of type 2. We define the following sets of indices:
\begin{itemize}
\item $I^{+}=\{i\in I\mid \varphi(e_{i})=e_{i}\}$,
\item $I^{-}=\{i\in I\mid \varphi(e_{i})=-e_{i}\},$ and 
\item $J=\{j\in\mathbb{N}\mid j\notin I\}.$
\end{itemize} 
Since $\varphi$ is of type 2 the set $I=I^{+}\cup I^{-}$ is infinite. Therefore we have to take into account the following five cases:
\begin{description}
\item{S1.} Both $|I^+|$ and $|I^-|$ are infinite. 
\item{S2.} $|I^+|$ is infinite but $|I^-|$ and $|J|$ are finite.
\item{S3.} $|I^+|$ is infinite, $|I^-|$ is finite and $|J|$ is infinite. 
\item{S4.} $|I^+|$ is finite, $|I^-|$ is infinite and $|J|$ is finite. 
\item{S5.} $|I^+|$ is finite, $|I^-|$ is infinite and $|J|$ is also infinite. 
\end{description}
Using the notation introduced above we have the following proposition. 

\begin{proposition}
The $T_{2}$-ideal of the graded identities of a Grassmann superalgebra of Case S1 coincides with $T_{2}(E_{\infty})$. 
\end{proposition}
\begin{proof}
Suppose that $\varphi$ is as in case S1, then  the sets $I^{+}$ and $I^{-}$ are infinite. Let $B$ be the subalgebra of $E_{\varphi}$ generated by $1$ and $e_{n}$, for every $n\in I^{+}\cup I^{-}$. In this case $B\subset E_{\varphi}$ is a homogeneous subalgebra and $B\simeq E_{\infty}$ (isomorphism of $\mathbb{Z}_{2}$-graded algebras). Hence $\langle[x_{1}, x_{2}, x_{3}]\rangle_{T_2}=T_{2}(E_{\infty})\supset T_{2}(E_{\varphi})$ where $\alpha(x_{i})=0$ or $\alpha(x_{i})=1$, for all $i=1$, 2, 3. Therefore it follows $T_{2}(E_{\varphi})=T_{2}(E_{\infty})$.
\end{proof}

Suppose now $E$ is equipped with a grading as in Case S2. Thus up to reordering the basis of $L$ the action of $\varphi$ on the generators is given by
\[
\varphi(e_{n})=\begin{cases} -e_{n},\text{ if }  1\leq n\leq k \\ -e_{n}+2a_{n}, \text{ if }  k+1\leq n\leq k+t \\ e_{n}, \text{ if }n>k+t 
\end{cases}.
\]
Here we write, as above, $a_j= (e_j+\varphi(e_j))/2$. We have that
\begin{itemize}
\item $E_{0,\varphi}\supset V_{0}=span_{F}\{1, e_{n}, a_{k+1},\ldots, a_{k+t}\mid n>k+t\}$. 
\item $E_{1,\varphi}\supset V_{1}=span_{F}\{ e_{1},\ldots,e_{k},(e_{k+1}-a_{k+1}),\ldots, (e_{k+t}-a_{k+t}) \}$.  
\end{itemize}
\begin{obs}\label{descricao precisa}
More precisely, the component $E_{0,\varphi}$ is generated by all products of elements in $V_{0}$ and $V_{1}$ with an even number of factors in $V_{1}$. The component $E_{1,\varphi}$ is generated by these products with an odd number of factors in $V_{1}$.
\end{obs}
Let $B$ be the subalgebra of $E_{\varphi}$ generated by the elements $e_{1}$, \dots, $e_{k}$, $e_{n}$, for every $n>k+t$.  Then $B$ is homogeneous and
\[
E_{k^\ast}\simeq B\subset E_{\varphi}.
\]
Moreover, each translation $(e_{k+1}-a_{k+1})$, \dots, $(e_{k+t}-a_{k+t})$ has zero square because the automorphism $\varphi$ is canonical. Thus it follows immediately that 
\[
z_{1}z_{2}\ldots z_{k+t+1}
\]
is a $\mathbb{Z}_{2}$-graded polynomial identity for $E_{\varphi}$.
Therefore we obtain
\[
T_{2}(E_{k^\ast})\supset T_{2}(E_{\varphi})\supset T_{2}(E_{(k+t)^\ast}).
\]

By imposing an additional condition we shall describe the $T_{2}$-ideal $T_{2}(E_{\varphi})$. Beforehand we state a theorem due to Anisimov's that will be important for our goals. 

\begin{theorem}[Anisimov, \cite{anisimov2}]\label{anisimov2-theorem}
	Let $\varphi$ be an automorphism of order 2 of $E$, and let $\varphi$ be  of canonical type.  Suppose that at least one of the following two conditions holds:
	\begin{enumerate}
		\item $\dim L_{-1}=l<\infty$ and $\displaystyle \prod\nolimits_{j=1}^{l+1}(\varphi(e_{i_j})-e_{i_j})=0$, for any $l+1$ generators $e_{i_1}$, \dots,  $e_{i_{l+1}}$.
		\item $\dim L_{1}=l<\infty$ and $\displaystyle \prod\nolimits_{j=1}^{l+1}(\varphi(e_{i_j})+e_{i_j})=0$, for any $l+1$ generators $e_{i_1}$, \dots,  $e_{i_{l+1}}$. 	
	\end{enumerate}
	Then one has that $T_{2}(E_{\varphi})=T_{2}(E_{\varphi_{l}})$. 	
\end{theorem}   

By using Anisimov's theorem \ref{anisimov2-theorem} and an additional we can describe the $\mathbb{Z}_{2}$-graded polynomial identities  for the Grassmann superalgebras in the Case S2.

\begin{theorem}\label{S2}
Let $\varphi$ be an automorphism of $E$ of type S2 and $T_{2}(E_{\varphi})$ its $T_{2}$-ideal of $\mathbb{Z}_{2}$-graded polynomial identities. If each element $a_{k+1}$,  \dots, $a_{k+t}$ is a linear combination of monomials of length $\geq 3$ then $T_{2}(E_{\varphi})=T_{2}(E_{\varphi_{l}})$.  
\end{theorem}

\begin{proof}
By hypothesis we know that $\varphi$ is of type S2 and therefore $I^{-}$ and $J$ are finite sets. Let us assume 
	\[|I^{-}|+ |J|=l.\]
	By Proposition \ref{auto tipo 2} we have that $\varphi$ is of the canonical type and $\dim L_{-1}=l<\infty$, and moreover $\displaystyle \prod\nolimits_{j=1}^{l+1}(\varphi(e_{i_j})-e_{i_j})=0$, for any choice of $l+1$  generators $e_{i_1}$, \dots, $e_{i_{l+1}}$. By Theorem \ref{anisimov2-theorem} it follows that $T_{2}(E_{\varphi})=T_{2}(E_{\varphi_{l}})$.  
\end{proof}

\begin{theorem}
Let $\varphi$ be an automorphism of $E$ of type S4 and $T_{2}(E_{\varphi})$ its $T_{2}$-ideal of $\mathbb{Z}_{2}$-graded polynomial identities. If each element $a_{k+1}$,  \dots, $a_{k+t}$ is a linear combination of monomials of length $\geq 3$ then $T_{2}(E_{\varphi})=T_{2}(E_{\varphi_{l}})$.   
\end{theorem}
\begin{proof}
	The proof repeats verbatim the one of Theorem \ref{S2}. 
\end{proof}

In the remaining cases we can not guarantee the equality but at least one of the inclusions still holds. 

\begin{theorem}
Let $\varphi$ be an automorphism of type 2 of $E$. The following statements hold: 
\begin{enumerate}
	\item If $\varphi$ satisfies  S3 then $T_{2}(E_{\varphi})\subseteq T_{2}(E_{k^\ast})$. 
	\item If $\varphi$ satisfies S5 then $T_{2}(E_{\varphi})\subseteq T_{2}(E_{k})$. 		
\end{enumerate}
\end{theorem}
\begin{proof}
The proof is similar to that of Theorem \ref{S2} and therefore we omit it. 
\end{proof}

These results show that the $\mathbb{Z}_{2}$-graded polynomials identities of many superalgebras of type 2 coincide with the identities of the homogeneous case. In other words one cannot distinguish these superalgebras by means of their graded polynomial identities.

\subsection{Concrete superalgebras of type 2}  
In this subsection we construct certain $\mathbb{Z}_{2}$-gradings of type 2. We shall prove that in many ``typical'' cases the superalgebras of type 2 are $\mathbb{Z}_{2}$-isomorphic to homogeneous superalgebras. In what follows we present a method to construct such structures. 

Choose an infinite set $I\subset\mathbb{N}$, $I\neq\mathbb{N}$, and define an action
$\varphi(e_{i})=\pm{e_{i}}$ whenever $i\in I$. As before we form the sets:
\begin{itemize}

\item $I^{+}=\{i\in I\mid\varphi(e_{i})=e_{i}\}$;
\item $I^{-}=\{i\in I\mid\varphi(e_{i})=-e_{i}\}$;
\item $J=\{j\in\mathbb{N}\mid j\notin I\}$.
\end{itemize}
Note that $I=I^{+}\cup I^{-}$, a disjoint union. For each $j\in J$, we construct the elements $d_{j}\in E$ in the following way: 
\begin{enumerate}
\item {\it The element $d_{j}$ is a linear combination of monomials of odd length.}
\item {\it All monomials that occur in $d_{j}$ are products of generators whose indices belong to $I$}.
\item {\it All monomials that occur in $d_{j}$ have an even numbers of factors in $I^{-}$}. 
\end{enumerate}
We extend the action of $\varphi$ on all $e_i$ as follows.
\[
\varphi(e_{i})=\begin{cases} \ e_{i},\text{ if } i\in I^{+} \\ -e_{i}, \text{ if } i\in I^{-} \\ -e_{i}+2d_{i},\text{ if } i\in J
\end{cases}
\]
Hence
\begin{itemize}
\item Condition (1) implies that $\varphi$ can be extended to an endomorphism of $E$, and
\item Conditions (2) and (3) imply that $\varphi(d_{j})=d_{j}$ for every $j\in J$. 
\end{itemize}

We claim that $\varphi$ is an automorphism of order 2. Indeed, if $i\in I$ we note that $\varphi^{2}(e_{i})=e_{i}$. If $j\in J$ we have
\[
\varphi^{2}(e_{j})=\varphi(-e_{j}+2d_{j})=-(-e_{j}+2d_{j})+2d_{j}=e_{j}.
\]
This ensures that $\varphi$ is an automorphism of order 2 of $E$. Let $E_{\varphi}=E_{0,\varphi}\oplus E_{1,\varphi}$ be the superalgebra defined by $\varphi$. Its homogeneous components are
\begin{itemize}
\item $E_{0,\varphi}\supset span_{F}\{e_{i}, d_{j}\mid i\in I^{+}, j\in J \}$ and
\item $E_{1,\varphi}\supset span_{F}\{e_{i}, (e_{j}-d_{j})\mid i\in I^{-}, j\in J\}$.	
\end{itemize}

We recall that the precise description for these components was done in Remark \ref{descricao precisa}. 
We call the previous method for obtaining a $\mathbb{Z}_{2}$-grading on $E$ the \textsl{method $1$}. For each automorphism $\varphi$ constructed by means of the method $1$, it is straightforward to deduce that that the superalgebra $E_{\varphi}$ is not homogeneous. Moreover if $J$ is infinite then there exist infinitely many generators of $E$ which are not homogeneous. 

In the following proposition we show that the construction from method $1$, while yielding \textsl{many} non homogeneous generators of $E$, produces superalgebras which do not differ significantly from the homogeneous ones.

\begin{proposition}
Let $\varphi\in Aut(E)$ be an automorphism of order 2 constructed by the method $1$. Then there exists a $\mathbb{Z}_{2}$-graded isomorphism between $E_{\varphi}$ and some of the superalgebras $E_{\infty}$, $E_{k^\ast}$ or $E_{k}$, for some $k$. 
\end{proposition}
\begin{proof}
Let $E_{\varphi}=E_{0,\varphi}\oplus E_{1,\varphi}$ be a $\mathbb{Z}_{2}$-grading on $E$ induced by some automorphism $\varphi$ constructed by method $1$. 

We define $E=E^{(0)}\oplus E^{(1)}$ the homogeneous $\mathbb{Z}_{2}$-grading induced by the following grading on the generators of $E$:
\begin{eqnarray*}
\|e_{n}\|=0,&&\textrm{ if  }n\in I^{+},\\
\|e_{m}\|=1,&&\textrm{ otherwise.}
\end{eqnarray*}
Clearly such a grading produces one of the superalgebras $E_{\infty}$, $E_{k^\ast}$ or $E_{k}$.
Now let $f_{\varphi}\colon E^{(0)}\oplus E^{(1)}\longrightarrow E_{\varphi}=E_{0,\varphi}\oplus E_{1,\varphi}$ be defined by
\[
f_{\varphi}(e_{i})=\begin{cases} e_{i},\text{ if } i\in I \\ e_{i}-d_{i}, \text{ if } i\in J
\end{cases}.
\]
Since the images of $f_\varphi$ satisfy the relation $x_{1}x_{2}+x_{2}x_{1}=0$ we can extend $f_{\varphi}$ to a homomorphism of $E$. By the definition of $f_\varphi$ we can see that it preserves the degrees of the generators, it is invertible and its inverse $g_{\varphi}\colon E_{\varphi}\longrightarrow E^{(0)}\oplus E^{(1)} $ is of the form
\[g_{\varphi}(e_{i})=\begin{cases} e_{i},\text{ if } i\in I \\ e_{i}+d_{i}, \text{ if } i\in J
\end{cases}.\]
Therefore $E_{\varphi}$ is isomorphic to one of the superalgebras $E_{\infty}$, $E_{k^\ast}$ or $E_{k}$. 
\end{proof}

\begin{example}
Let $I=I^{+}=\{2,3,4,\ldots\}$ and define the automorphism $\varphi$ by its action on the generators of $E$:
\[
\varphi(e_{i})=\begin{cases} e_{i},\text{ if } i\neq 1 \\ -e_{1}+2e_{2}e_{3}e_{4}, \text{ if } i=1
\end{cases}.
\]
In this case the grading is 
\begin{itemize}
\item $E_{0,\varphi}\supset span_{F}\{e_{2}, e_{3}, e_{4},\ldots \}$.
\item $E_{1,\varphi}\supset span_{F}\{e_{1}-e_{2}e_{3}e_{4} \}$.
\end{itemize}
Consider $E$ endowed with the $\mathbb{Z}_{2}$-grading produced by the method presented in the previous proposition. In other words, $\|e_{1}\|=1$ and $\|e_{i}\|=0$ for $i\neq 1$ (it is exactly the superalgebra $E_{1^\ast}$). Hence the action $f_{\varphi}\colon E\to E_{\varphi}$, given by
\[f_{\varphi}(e_{i})=\begin{cases} e_{i},\text{ if } i\neq 1 \\ e_{1}-e_{2}e_{3}e_{4}, \text{ if } i=1
\end{cases},\]
can be extended to a $\mathbb{Z}_{2}$-graded isomorphism between $E_{1^\ast}$ and $E_{\varphi}$. 
Consequently, 
\[
T_{2}(E_{\varphi})=\langle [x_{1},x_{2},x_{3}], z_{1}z_{2}\rangle_{T_{2}} .
\]
On the other hand we can define another action as follows:
\[
\psi(e_{i})=\begin{cases} e_{i},\text{ if } i\neq 1 \\ -e_{1}+2(e_2+e_3+e_4e_5e_6 +e_3e_4e_7e_8e_9), \text{ if } i=1
\end{cases}.
\]
Note that $\psi$ yields the superalgebra 
\begin{itemize}
\item $E_{0,\psi}\supset span_{F}\{ e_{2}, e_{3}, e_{4},\ldots \}$.
\item $E_{1,\psi}\supset span_{F}\{ e_{1}-(e_2+e_3+e_4e_5e_6 +e_3e_4e_7e_8e_9) \}$.
\end{itemize}
By the previous proposition we obtain that 
\[T_{2}(E_{\varphi})=T_{2}(E_{\psi})=T_{2}(E_{1^\ast}).\]
\end{example}

\subsection{Triangular $\mathbb{Z}_{2}$-gradings on $E$} 

In this subsection we describe a particular way of constructing $\mathbb{Z}_{2}$-gradings on $E$. We call these gradings \textsl{triangular $\mathbb{Z}_{2}$-gradings}. For each $N\in\mathbb{N}$, we shall construct a subgroup of $Aut(E)$ of order $2^{N}$ which has the property that all its elements induce Grassmann superalgebras.

Recall that we denote by $E_{can}=E_{(0)}\oplus E_{(1)}$ the natural $\mathbb{Z}_{2}$-grading of $E$; if $k\in\mathbb{N}$, we denote:
\[
E(e_{1},\ldots, e_{k})=\text{ the span of all monomials without factors in } \{e_{1},\ldots, e_{k}\}.
\]

\begin{definition}
	For every $n\in\mathbb{N}$, the automorphism $T_{n}\colon E\to E$ defined by
\[
T_{n}(e_{j})=\begin{cases} e_{j},\text{ if } j\neq n,\\ -e_{j}+2P_{j},\text{ if }j=n,\end{cases}
\]
where $P_{n}\in E(e_{1},\ldots, e_{n})\cap E_{(1)}$, is called a triangular automorphism of index $n$. 
	 \end{definition}

\begin{obs}
It is easy to prove that each triangular automorphism
\begin{itemize}
	\item is of order 2,
	\item can be constructed by the method $1$,
	\item induces a Grassmann superalgebra that is $\mathbb{Z}_{2}$-isomorphic to $E_{1^\ast}$.
\end{itemize}	
\end{obs}

Take $N\in\mathbb{N}$ and let $T_{1}$, \dots, $T_{N}$ be triangular automorphisms such that 
\[
P_{j}\in E(e_{1},\ldots, e_{N})\cap E_{(1)}
\]
for every $j=1$, \dots, $N$. 

Denote by $\tau_{N}$ the subgroup of $Aut(E)$ generated by $T_{1}$, \dots, $T_{N}$, that is 
\[
\tau_{N}=\langle T_{1},\ldots, T_{N} \rangle.
\]

\begin{proposition}
	Let $\tau_{N}$ be the subgroup of $Aut(E)$ constructed above. Then we have that:
	\begin{enumerate}
		\item $\tau_{N}$ is an abelian group.
		\item $\tau_{N}$ is of order $2^{N}$.
		\item Each element of $\tau_{N}$ induces a Grassmann superalgebra.
		\item If $\varphi\in\tau_{N}$ and $\varphi=T_{j_1}\circ\cdots\circ T_{j_s}$, with $j_{i}\neq j_{l}$, for $i\neq l$, then $E_{\varphi}$ is $\mathbb{Z}_{2}$-isomorphic to $E_{s^\ast}$.
	\end{enumerate}
\end{proposition}
\begin{proof}
	We start with the proof of statement 1. Let $u$, $v\in\{1,\ldots, N\}$ with $u<v$. If $j\neq u$, $j\ne v$, then clearly
\[
T_{u}\circ T_{v}(e_{j})=T_{v}\circ T_{u}(e_{j})=e_{j}.
\]
Moreover 
	\[
	T_{u}\circ T_{v}(e_{u})=T_{u}(e_{u})=-e_{u}+2P_{u}, \quad 
T_{v}\circ T_{u}(e_{u})=T_{v}(-e_{u}+2P_{u})=-e_{u}+2P_{u}
\]
hence $T_{u}\circ T_{v}(e_{u})=T_{v}\circ T_{u}(e_{u})$. In the same way we compute that $T_{u}\circ T_{v}(e_{v})=T_{v}\circ T_{u}(e_{v})$, proving that $\tau_{N}$ is an abelian group. Statement 2 is an immediate consequence of statement 1. 
    
Now we prove statement 3. If $\varphi\in\tau_{N}$, it is enough to show that $\varphi^{2}=1$. Let  
	 \[\varphi=T_{j_1}\circ\cdots\circ T_{j_s}.\]
	 Obviously $\varphi(e_{n})=e_{n}$ whenever $n\notin\{j_{1},\ldots, j_{s}\}$. On the other hand if $n=j_{t}$ for some $t=1$, \dots, $s$, it is easy to see that:
	 \[
	 \varphi(e_{j_{t}})=-e_{j_{t}}+2P_{j_{t}}
	 \]
	 and we get $\varphi^{2}(e_{j_{t}})=e_{j_{t}}$.

In order to prove Statement 4 one applies the argument used in the case of method 1.
\end{proof}

\section{Gradings and automorphisms of type 3}
We focus our attention on automorphisms of type 3. Some of the statements are similar to results of the previous section. We start with the following  proposition. 

\begin{proposition}\label{the aut 3}
Let $\varphi\in Aut(E)$ be an automorphism of order 2 and suppose that $\varphi$ is  of type 3, that is, $I_{\beta}=\{1,\ldots,k\}$ where $k\geq 1$. Then one of the following two possibilities holds:
\begin{itemize}
		\item either $\varphi$ is of the canonical type, or 
		\item $\varphi$ is defined by
		$\displaystyle\varphi(e_{n})=\begin{cases}\pm e_{n},\text{ if } 1\leq n\leq k \\ -e_{n}+2e_{1}\cdots e_{k}V_{n}+2W_{n}, \text{ if } n>k
		\end{cases}.$
\end{itemize}       
Here the set $\{V_{n}\}_{n > k}$ consists of nonzero elements, its elements are of the same parity as that of $k$, each $V_{n}$ has no summands with factors among $e_{1}$, \dots,  $e_{k}$, and the set $\{W_{n}\}_{n > k}$  is formed by elements of odd length. Moreover the following relations hold:
		\begin{equation}
		e_{1}\cdots e_{k}(V_{p}e_{r}+V_{r}e_{p})=2e_{1}\cdots e_{k}(V_{p}W_{r}+V_{r}W_{p}), \text{ for all $r$, $p>k$} 
		\label{eq1}
		\end{equation}     
\end{proposition}

\begin{proof}
As in the previous arguments we can write, for $n>k$,
\[
\varphi(e_{n})=-e_{n}+2a_{n}
\]
where each $a_{n}=(e_n+\varphi(e_n))/2$ is invariant under $\varphi$, that is $\varphi(a_n)=a_n$.  Once again we denote $a_{n}=a_{n}^{e}+a_{n}^{o}$ where $a_{n}^{e}$ and $a_{n}^{o}$ are the even and the odd part of $a_{n}$, respectively.
  For each $t\in\{1,\ldots,k\}$ and for each $n > k$, we have $e_{t}e_{n}+e_{n}e_{t}=0$. 
Therefore $\varphi(e_{t})\varphi(e_{n})+\varphi(e_{n})\varphi(e_{t})=0$ and thus $e_{t}a_{n}^{e}=0$. 
 
Hence either $a_{n}^{e}=0$ or each monomial in the part $a_{n}^{e}$ has $e_{t}$ as a factor. In both cases, since $t=1$, \dots, $k$ we obtain 
	\[
	a_{n}^{e}=e_{1}\cdots e_{k}V_{n}
	\]
where either $V_{n}=0$ or the monomials in $V_{n}$ have no factors among $e_{1}$, \dots, $e_{k}$, and all of  them are of even length if $k$ is even, or of odd length, whenever $k$ is odd.

Now we consider $n > k$ and define $W_{n}=a_{n}^{o}$. In this case 
	\[
	\varphi(e_{n})=-e_{n}+2e_{1}\cdots e_{k}V_{n}+2W_{n}
\]
where each $W_{n}$ is a linear combination of odd monomials and $V_{n}$ is a linear combination of even or odd monomials.  

If $p$, $r>k$, we have 
	\[
	\varphi(e_{p})\varphi(e_{r})+\varphi(e_{r})\varphi(e_{p})=0
	\]
and then
\begin{eqnarray*}
&&(-e_{p}+2e_{1}\cdots e_{k}V_{p}+2W_{p})(-e_{r}+2e_{1}\cdots e_{k}V_{r}+2W_{r})\\
&&+(-e_{r}+2e_{1}\cdots e_{k}V_{r}+2W_{r})(-e_{p}+2e_{1}\cdots e_{k}V_{p}+2W_{p})=0.
\end{eqnarray*}    
 Hence 
 \begin{eqnarray*}
&&-2e_{p}e_{1}\cdots e_{k}V_{r}-2e_{1}\cdots e_{k}V_{p}e_{r}+4e_{1}\cdots e_{k}V_{p}W_{r}+4W_{p}e_{1}\cdots e_{k}V_{r}\\
&&-2e_{r}e_{1}\ldots e_{k}V_{p}-2e_{1}\cdots e_{k}V_{r}e_{p}+4e_{1}\cdots e_{k}V_{r}W_{p}+4W_{r}e_{1}\cdots e_{k}V_{p}=0
\end{eqnarray*}
and in this way
	\[
	-4e_{p}e_{1}\cdots e_{k}V_{r}-4e_{r}e_{1}\cdots e_{k}V_{p}+8e_{1}\cdots e_{k}V_{p}W_{r}+8e_{1}\cdots e_{k}V_{r}W_{p}=0.
	\]
All these computations yield the relation:
\[
e_{1}\cdots e_{k}(V_{p}e_{r}+V_{r}e_{p})=2e_{1}\cdots e_{k}(V_{p}W_{r}+V_{r}W_{p}), \text{ for all $p$, $r>k$}.
\]
We have to examine separately the following two cases:
\begin{itemize}
\item[Case 1.] If $V_{n}=0$, for all $n>k$, then clearly $\varphi$ is of canonical type.	
\item[Case 2.]
Otherwise we have that $\varphi$ is defined by
\[\varphi(e_{n})=\begin{cases}\pm e_{n},\text{ if } 1\leq n\leq k \\ -e_{n}+2e_{1}\cdots e_{k}V_{n}+2W_{n}, \text{ if } n > k
	\end{cases}\]
where $\{V_{n}\}_{n > k}$ and $\{W_{n}\}_{n > k}$ both satisfy the conditions stated in the proposition. Furthermore the  following relations hold
\[
e_{1}\cdots e_{k}(V_{p}e_{r}+V_{r}e_{p})=2e_{1}\cdots e_{k}(V_{p}W_{r}+V_{r}W_{p}).
\]
%\label{eq1}
	\end{itemize}
Therefore the proof is now complete.   
\end{proof}

\subsection{Concrete superalgebras of type 3}

Now we exhibit a method that provides us with automorphisms of $E$ of type 3. We shall see that they need not be necessarily of canonical type.

Let $t$, $k$ be integers such that $t$ is odd and $k\in\{0,1,\ldots\}$. We shall denote by $I$ the set $I=\{1,\ldots,k,k+1,\ldots, k+t\}\subset\mathbb{N}$. For each $i\in I$  we define the action $\varphi(e_{i})=\pm e_{i}$ so that $I^{+}=\{1,\ldots,k\}$ and $I^{-}=\{k+1,\ldots,k+t\}$. For $n>k+t$ we put
\[
\varphi(e_{n})=-e_{n}+2e_{1}\cdots e_{k}e_{k+1}\cdots e_{k+t}e_{n}.
\]
Hence the action of $\varphi$  on the generators is given by 
\[
\varphi(e_{n})=\begin{cases} e_{n},\text{ if } 1\leq n\leq k \\ -e_{n}, \text{ if } k+1\leq n\leq k+t \\ -e_{n}+2e_{1}\cdots e_{k}e_{k+1}\cdots e_{k+t}e_{n}, \text{ otherwise }
\end{cases}.
\]
Now we note that
\begin{itemize}
	\item $\varphi$ preserves the relations $e_{i}e_{j}+e_{j}e_{i}=0$ and so it can be extended to a homomorphism of $E$.
	\item For $n>k+t$
	\begin{eqnarray*}
\varphi^{2}(e_{n})&=&\varphi(-e_{n}+2e_{1}\cdots e_{k}e_{k+1}\cdots e_{k+t}e_{n})\\
	&=&-(-e_{n}+2e_{1}\cdots e_{k}e_{k+1}\cdots e_{k+t}e_{n})\\
	&&+2e_{1}\cdots e_{k}(-1)^{t}e_{k+1}\cdots e_{k+t}(-e_{n}+2e_{1}\cdots e_{k}e_{k+1}\cdots e_{k+t}e_{n}).
	\end{eqnarray*}
	
Therefore $\varphi^{2}(e_{n})=e_{n}$.
\end{itemize}

Thus $\varphi$ is an automorphism of order 2 on $E$. The induced superalgebra  $E_{\varphi}=E_{0,\varphi}\oplus E_{1,\varphi}$ has its homogeneous components described by
\begin{itemize}
	\item $E_{0,\varphi}\supset span_{F}\{e_{1},\ldots, e_{k}\},$
	\item $E_{1,\varphi}\supset span_{F}\{e_{k+1},\ldots, e_{k+t}, (e_{n}-e_{1}\cdots e_{k}e_{k+1}\cdots e_{k+t}e_{n})\mid n>k+t\}.$
\end{itemize}
We call the method just described the \textsl{method $2$} for obtaining Grassmann superalgebras. Whenever $k$ is an even integer, the automorphism $\varphi$ is not of canonical type. Nevertheless, as we shall see below, the  superalgebras obtained in this way will be isomorphic to homogeneous ones. 

\begin{proposition}
Let $E_{\varphi}$ be a $\mathbb{Z}_{2}$-grading obtained according to method $2$. Then $E_{\varphi}$ and $E_{k}$ are isomorphic as $\mathbb{Z}_{2}$-graded algebras.  
\end{proposition}

\begin{proof}
	Let $f_{\varphi}\colon E_{k}\longrightarrow E_{\varphi}$ be the function defined on the generators by
		\[f_{\varphi}(e_{n})=\begin{cases} e_{n},\text{ if } 1\leq n\leq k \\ e_{n}, \text{ if } k+1\leq n\leq k+t \\ e_{n}-e_{1}\cdots e_{k}e_{k+1}\cdots e_{k+t}e_{n}, \text{ otherwise }
	\end{cases}.\]
It is easy to check that $f_{\varphi}$ preserves the relations $e_{i}e_{j}+e_{j}e_{i}=0$ and thus we can extend $f_{\varphi}$ to an endomorphism of the algebra $E$. Moreover
\begin{eqnarray*}
f_{\varphi}(e_{1}\cdots e_{k}e_{k+1}\cdots e_{k+t}e_{n}) &=& e_{1}\cdots e_{k}e_{k+1}\cdots e_{k+t}f_{\varphi}(e_{n})\\
&=&e_{1}\cdots e_{k}e_{k+1}\cdots e_{k+t}(e_{n}-e_{1}\cdots e_{k}e_{k+1}\cdots e_{k+t}e_{n})\\
&=&e_{1}\cdots e_{k}e_{k+1}\cdots e_{k+t}e_{n}.
\end{eqnarray*}
Now we consider the homomorphism $g_{\varphi}\colon E_{\varphi}\to E_{k}$ defined by
\[
g_{\varphi}(e_{n})=\begin{cases} e_{n},\text{ if } 1\leq n\leq k \\ e_{n}, \text{ if } k+1\leq n\leq k+t \\ e_{n}+e_{1}\cdots e_{k}e_{k+1}\cdots e_{k+t}e_{n}, \text{ otherwise }
\end{cases}.
\]
By using the same argument as above we obtain 
\[
g_{\varphi}(e_{1}\cdots e_{k}e_{k+1}\cdots e_{k+t}e_{n})=e_{1}\cdots e_{k}e_{k+1}\cdots e_{k+t}e_{n}.
\]
Therefore if $n>k+t$ we conclude that
\begin{eqnarray*}
f_{\varphi}\circ g_{\varphi}(e_{n})&=&f_{\varphi}(e_{n}+e_{1}\cdots e_{k}e_{k+1}\cdots e_{k+t}e_{n})\\
&=&(e_{n}-e_{1}\cdots e_{k}e_{k+1}\cdots e_{k+t}e_{n})+e_{1}\cdots e_{k}e_{k+1}\cdots e_{k+t}e_{n}=e_{n}.
\end{eqnarray*}
We also obtain that $g_{\varphi}\circ f_{\varphi}(e_{n})=e_{n}$, and thus we prove that $f_{\varphi}$ is an automorphism of $E$. By construction the automorphism $f_{\varphi}$ is $\mathbb{Z}_{2}$-graded. Therefore the superalgebras $E_{\varphi}$ and $E_{k}$ are $\mathbb{Z}_{2}$-isomorphic, and our proof is complete. 
\end{proof}

\begin{proposition}
\label{prop_minus}
Let $\varphi\colon E\to E$ be the homomorphism defined by its action on the generators of $E$:	
	\[
	\varphi(e_{n})=\begin{cases} -e_{1},\text{ if }  n=1 \\ -e_{n}+2e_{1}e_{n}, \text{ if }  n>1
	\end{cases}.
	\]
Then $\varphi$ is an automorphism of order 2, and $E_{\varphi}$ and $E_{can}$ are $\mathbb{Z}_{2}$-isomorphic. 
\end{proposition}	
\begin{proof}
The proof follows from the construction of method $2$.
\end{proof}
\begin{obs}
\noindent Let $E_{\varphi}=E_{0,\varphi}\oplus E_{1,\varphi}$ be the superalgebra defined in Proposition \ref{prop_minus}. It is immediate that its homogeneous components are as follows:
	\begin{enumerate}
		\item $E_{1,\varphi}\supset span_{F}\{e_{1}, (e_{n}-e_{1}e_{n}) \mid n>1\}$,
		\item $E_{0,\varphi}=Z(E)$.
	\end{enumerate}
In this case just one of the generators, namely $e_1$, of $E$ is homogeneous. Moreover the automorphism $\varphi$ is not of canonical type. However Proposition \ref{prop_minus} guarantees that $E_{\varphi}$ is isomorphic to the natural $\mathbb{Z}_{2}$-grading $E_{can}$. 
\end{obs}

\section{Gradings and automorphisms of type 4}
In this section we consider automorphisms of type $4$. Recall that these are automorphisms such that for \textsl{every} basis $\beta$ of the underlying vector space $L$, one has $I_\beta=\emptyset$. Here $I_\beta$ are the eigenvectors of $\varphi$ belonging to $\beta$. (We recall that one may change the basis of $L$ so that $\varphi(e_n)=\pm e_n+$ a linear combination of monomials of higher degree.)  We shall prove that such structures do not exist in quite many cases.

\begin{proposition}
	There do not exist automorphisms of type 4 defined by
	\[
	\varphi(e_{n})=-e_{n}+b_{n}
\]
	where the $b_{n}\in E$ are nonzero monomials for every $n\in\mathbb{N}$, possibly multiplied by some (nonzero) scalars.
\end{proposition}
\begin{proof}
Let us suppose that for every $n\in\mathbb{N}$, one has
\[
\varphi(e_{n})=-e_{n}+b_{n}
\]
where $b_{n}=e_{n_{1}}\cdots e_{n_{k_{n}}}\neq 0$ is a monomial such that $\varphi(b_{n})=b_{n}$ for every $n\in\mathbb{N}$.
	 
Since $\varphi$ is of type 4, it follows that $b_{n}$ is of length $k_{n}=|b_{n}|>1$, for all $n\in\mathbb{N}$. Furthermore 
\[
e_{n_{1}}\cdots e_{n_{k_{n}}}=\varphi(b_{n})=(-e_{n_{1}}+b_{n_{1}})\cdots (-e_{n_{k_{n}}}+b_{n_{k_{n}}}).
\]
Hence one obtains immediately that 
\[
e_{n_{1}}\cdots e_{n_{k_{n}}}=(-1)^{k_{n}}e_{n_{1}}\cdots e_{n_{k_{n}}}+\sum_{|m_{n}|>k_{n}} {m_{n}}.
\]
It follows that $k_{n}$ is an even integer and therefore $b_{n}\in Z(E)$.

On the other hand, due to the equality $(-e_{n}+b_{n})^{2}=0$, we conclude that $e_{n}$ must be a factor of $b_{n}$. We may assume that $e_{n}$ occurs in the first position of the monomial $b_n$. Rearranging the indices and with certain abuse of notation we can write
\[
\varphi(e_{n})=-e_{n}+e_{n}e_{n_{1}}\cdots e_{n_{k_{n}}}
\]
where $k_{n}$ is an odd integer. 

Consequently by direct computation we obtain 
\begin{eqnarray*}
\varphi(e_{n})\varphi(e_{m})&=&(-e_{n}+e_{n}e_{n_{1}}\cdots e_{n_{k_{n}}})(-e_{m}+e_{m}e_{m_{1}}\cdots e_{m_{k_{m}}})\\
&=&e_{n}e_{m}-e_{n}e_{m}e_{m_{1}}\cdots e_{m_{k_{m}}}-e_{n}e_{n_{1}}\cdots e_{n_{k_{n}}}e_{m} \\
&&+2e_{n}e_{n_{1}}\cdots e_{n_{k_{n}}}e_{m}e_{m_{1}}\cdots e_{m_{k_{m}}},
\end{eqnarray*}
and analogously 
\begin{eqnarray*}
\varphi(e_{m})\varphi(e_{n})&=& (-e_{m}+e_{m}e_{m_{1}}\cdots e_{m_{k_{m}}})(-e_{n}+ e_{n}e_{n_{1}}\cdots e_{n_{k_{n}}})\\
&=&e_{m}e_{n}-e_{m}e_{n}e_{n_{1}}\cdots e_{n_{k_{n}}}-e_{m}e_{m_{1}}\cdots e_{m_{k_{m}}}e_{n}  \\
&&+2e_{m}e_{m_{1}}\cdots e_{m_{k_{m}}}e_{n}e_{n_{1}}\cdots e_{n_{k_{n}}}.
\end{eqnarray*}
Then the equality $\varphi(e_{n})\varphi(e_{m})+ \varphi(e_{m})\varphi(e_{n})=0$ implies 
\begin{eqnarray*}
&&-e_{n}e_{m}e_{m_{1}}\cdots e_{m_{k_{m}}}- e_{n}e_{n_{1}}\cdots e_{n_{k_{n}}}e_{m}+ 2e_{n}e_{n_{1}}\cdots e_{n_{k_{n}}}e_{m}e_{m_{1}}\cdots e_{m_{k_{m}}}\\
&&- e_{m}e_{n}e_{n_{1}}\cdots e_{n_{k_{n}}}- e_{m}e_{m_{1}}\cdots e_{m_{k_{m}}}e_{n}+ 2e_{m}e_{m_{1}}\cdots e_{m_{k_{m}}}e_{n}e_{n_{1}}\cdots e_{n_{k_{n}}}=0.
\end{eqnarray*}
Therefore
\[
-e_{n}e_{m}e_{m_{1}}\cdots e_{m_{k_{m}}}-e_{m}e_{n}e_{n_{1}}\cdots e_{n_{k_{n}}}+ 2e_{m}e_{m_{1}}\cdots e_{m_{k_{m}}}e_{n}e_{n_{1}}\cdots e_{n_{k_{n}}}=0.
\]
Let us fix the integer $m$ and choose $n\notin\{m_{1},\ldots, m_{k_{1}}\}$. Then we conclude that 
\[
e_{n_{1}}\cdots e_{n_{k_{n}}}=e_{m_{1}}\cdots e_{m_{k_{1}}}
\]
for every $n\notin\{m_{1},\ldots, m_{k_{1}}\}$.

In this case we have  $\varphi(e_{m_{1}})=-e_{m_{1}}+e_{m_{1}}P$, \dots,  $\varphi (e_{m_{k_1}}) = -e_{m_{k_1}}+e_{m_{k_1}}Q$, and $\varphi(e_{n})=-e_{n}+e_{n}e_{m_{1}}\cdots e_{m_{k_1}}$, for every $n\notin\{1_{1},\ldots, 1_{k_{1}}\}$ (here we denote by$P$, \dots, $Q$ some monomials in $E$).

Now we fix $m=m_1$. Choosing $n\notin\{m_{1},\ldots, m_{k_{1}}\} $ with $e_{n}\notin supp(P)$, and using the same idea as in the case $m=1$, we obtain	
\[e_{n_{1}}\cdots e_{n_{k_{n}}}=P.
\]
Therefore we conclude that $e_{n_{1}}\cdots e_{n_{k_{n}}}$ satisfies the equalities 
\[
P=e_{n_{1}}\cdots e_{n_{k_{n}}}=e_{m_{1}}\cdots e_{m_{k_{1}}}.
\]
Recall that $m_1$ is fixed and there exist infinitely many $n$ as above. Thus we conclude that $\varphi(e_{m_{1}})=-e_{m_{1}}+e_{m_{1}}P=-e_{m_{1}}+e_{m_{1}}(e_{m_{1}}\cdots e_{m_{k_{1}}})=-e_{m_{1}}$, which is a contradiction. The proposition is proved.
\end{proof}

Another way of constructing automorphisms of type 4 could be by translation by a constant. These are automorphisms defined by $\varphi(e_{n})=-e_{n}+2P$ where $P\in E$ is constant and contains some summand of length $\geq 2$. However, as we shall see in the next proposition, this also turns out to be impossible. 

\begin{proposition}
Let $\varphi$ be an automorphism of type 4 defined by the following action on the generators of $E$:
	\[
	\varphi(e_{n})=-e_{n}+2a_{n}, \qquad a_n=(e_n+\varphi(e_n))/2.
	\]
Then the set $T=\{a_{n}\mid n\in\mathbb{N}\}$ is linearly independent over $F$. 
\end{proposition}

\begin{proof}
	Let $k\in\mathbb{N}$ be a positive integer, $a_{i_{1}}$, \dots, $a_{i_{k}}\in T$ and $\lambda_{1}$, \dots, $\lambda_{k}\in F$ and assume that
	\[
	\lambda_{1}a_{i_{1}}+\cdots + \lambda_{k}a_{i_{k}}=0.
	\]
Then we have
\begin{eqnarray*}
\varphi(\lambda_{1}e_{i_{1}}+\cdots + \lambda_{k}e_{i_{k}})&=& \lambda_{1}(-e_{i_{1}}+2a_{i_{1}})+\cdots + \lambda_{k}(-e_{i_{k}}+2a_{i_{k}})\\ 
&=&-(\lambda_{1}e_{i_{1}}+\cdots + \lambda_{k}e_{i_{k}})+2(\lambda_{1}a_{i_{1}}+\cdots + \lambda_{k}a_{i_{k}}).
\end{eqnarray*}
In this way we obtain
	\[
	\varphi(\lambda_{1}e_{i_{1}}+\cdots + \lambda_{k}e_{i_{k}})=-(\lambda_{1}e_{i_{1}}+\cdots + \lambda_{k}e_{i_{k}}).
	\]
If the linear combination $\lambda_{1}e_{i_{1}}+\cdots + \lambda_{k}e_{i_{k}}$ does not vanish we can choose a basis $\gamma$ of the vector space $L$ containing the vector $\lambda_{1}e_{i_{1}}+\cdots + \lambda_{k}e_{i_{k}}$ (say as a first vector of the basis). But this is a contradiction because $\varphi$ is of  type 4. 
	
	Therefore $\lambda_{1}e_{i_{1}}+\cdots + \lambda_{k}e_{i_{k}}=0$ and $\lambda_{1}=\cdots= \lambda_{k}=0$. This proves that the set $T$ is linearly independent. 	
\end{proof}

\begin{corollary}
There does not exist automorphism of type 4 which is a translation by a constant.
\end{corollary}

\begin{proof}
The corollary follows immediately from the above proposition.  
\end{proof}

The results we have obtained in this section show that in many situations the fourth structure of a Grassmann superalgebra cannot exist. We believe we have grounds to conjecture that this is true in the general case.

\begin{conjectura}\label{conjtipo4}
Let $E=\varepsilon_{0}\oplus \varepsilon_{1}$ be an arbitrary $\mathbb{Z}_{2}$-grading on the Grassmann algebra. Then there exists a basis of the underlying vector space $L$ that contains at least one homogeneous generator of $E=\varepsilon_{0}\oplus \varepsilon_{1}$.   
\end{conjectura}

We shall describe another situation where one cannot obtain an automorphism of type 4, thus giving an extra strength to our conjecture. Recall that if $\varphi\colon E\to E$ is an automorphism of order $n$ where $n$ is a positive integer, the linearization $\varphi_\ell$ of $\varphi$ is defined as follows. Assume $\varphi(e_i) = u_i + v_i$ where $u_i\in L$ and $v_i$ is a linear combination of monomials of degrees $\ge 2$. Then $\varphi_\ell (e_i) = u_i$, and we extend then $\varphi_\ell$ from the $e_i$ to a homomorphism of $E$. Anisimov proved (see for example \cite{anisimov1, anisimov2}) that if $\varphi$ is an automorphism the so is $\varphi_\ell$. Moreover if $\varphi$ is of order 2 then $\varphi_\ell$ is also of order 2. 

We assume $\varphi$ is an automorphism of type 4, and choose a basis $\{e_i\}$ of $L$ such that $\varphi_\ell(e_i) =\pm e_i$ for every $i$ (this is always possible since $\varphi_\ell$ is also of order 2). 

We shall prove that $\varphi_\ell$ cannot be the identity map on $E$. If it were then $\varphi_\ell(e_i) = e_i$ for each $i$. We write $\varphi(e_i) = e_i + v_i$ where $v_i$ is a linear combination of monomials of lengths $\ge 2$, and let $s$ be the least length of a monomial that appears with non-zero coefficient in some of the $\{v_i\mid i\in\mathbb{N}\}$. Take an index $j$ such that $\varphi(e_j) = e_j+ \beta e_{j_1} \cdots e_{j_s} + w$ where $\beta\ne 0$ is in $F$ and $w$ is a linear combination of monomials of lengths $\ge s$. Then we have $\varphi^2(e_j)=e_j$ and this can be written as 
\begin{eqnarray*}
e_j&=& \varphi^2(e_j) = \varphi(e_j+ \beta e_{j_1} \cdots e_{j_s} + w) \\
&=&(e_j+ \beta e_{j_1} \cdots e_{j_s} + w) + \beta\varphi(e_{j_1})\cdots \varphi(e_{j_s}) + \varphi(w).
\end{eqnarray*}
Clearly $\varphi(e_{j_1})\cdots \varphi(e_{j_s}) = e_{j_1} \cdots e_{j_s} +$ some terms of lengths $\ge s$ all of them different from $e_{j_1} \cdots e_{j_s}$. 

Therefore we obtain $e_j = e_j + 2\beta e_{j_1} \cdots e_{j_s} +$ terms of lengths $\ge s$, all of them different from $e_{j_1} \cdots e_{j_s}$. Hence $\beta=0$, a contradiction. In this way we have proved the following theorem.

\begin{theorem}
\label{lin_id}
There is no automorphism $\varphi$ of type 4 such that $\varphi_\ell$ is the identity on $E$. 
\end{theorem}

\begin{corollary}
If $\varphi$ is an automorphism of type 4 then the eigenspace of $\varphi_\ell$ in $L$ corresponding to the eigenvalue $-1$ is non-zero. 
\end{corollary}

The following example is also related to our conjecture.  
\begin{example}\label{example conjecture}
Let us fix a basis $\beta=\{e_{1}, e_{2},\ldots,e_{n}\ldots\}$ of the vector space $L$. We define the action of the automorphism $\varphi$ on the basis $\beta$ as follows
\[
\varphi(e_{2i-1})=e_{2i}, \qquad \varphi(e_{2i})=e_{2i-1},  
\]
for every $i\in\mathbb{N}$.

Note that there does not exist any element belonging to $\beta$ which is invariant under $\varphi$. However it is easy to choose another basis of $L$ which contains invariant  elements. For example instead of $e_1$ and $e_2$ in $\beta$ one takes $e_1+e_2$ and $e_1-e_2$, leaving the remaining elements of the basis $\beta$. It is exactly to this kind of property that Conjecture \ref{conjtipo4} refers to. 
\end{example} 
%We summarize the superalgebras that were studied in this paper:
%\[
%\begin{tabular}{|c|c|}
%\hline
%$\varphi$ of the type & $2$ \\ \hline
%$S_{1}$ & $T_{2}(E_{\varphi})=T_{2}(E_{\infty})$ \\ \hline
%$S_{2}$ & $T_{2}(E_{\varphi})$ coincides with some homogeneous case \\ \hline 
%$S_{3}$ & $T_{2}(E_{\varphi})\subset T_{2}(E_{k^\ast})$ \\ \hline 
%$S_{4}$ & $T_{2}(E_{\varphi})$ coincides with some homogeneous case \\ \hline
%$S_{5}$ & $T_{2}(E_{\varphi})\subset T_{2}(E_{k})$\\ \hline 
%Method $A_1$ & $E_{\varphi}$ is $\mathbb{Z}_{2}$-isomorphic to some homogeneous case\\ \hline  
%subgroup $\tau_{N}$ & $E_{\varphi}\cong_{\mathbb{Z}_2} E_{s^\ast}, \text{ for every }\varphi\in\tau_{N}$ \\ \hline
%\end{tabular}
%\]
%\[
%\begin{tabular}{|c|c|}
%\hline
%$\varphi$ of the type & $3$ \\ \hline
%Method $A_2$ & $E_{\varphi}$ is $\mathbb{Z}_{2}$-isomorphic to some homogeneous case\\ \hline 
%\end{tabular}
%\]

%\begin{tabular}{|c|c|}
%\hline
%$\varphi$ of the type & $4$ \\ \hline
%By monomial translation & there does not exist automorphism $\varphi$ \\ \hline 
%By translation by a constant & there does not exist automorphism $\varphi$ 
%\\ \hline 
%$\varphi_\ell$ is the identity & there does not exist automorphism $\varphi$ 
%\\ \hline 
%\end{tabular}

\end{document}